\documentclass[11pt]{article}
\usepackage{latexsym,bm}
\usepackage{mathrsfs}
\usepackage{amsmath,amsthm}
\usepackage{graphicx}
\usepackage{amssymb}
\usepackage{CJK}
\usepackage{color}
\usepackage{cite}
\usepackage{bm}
\usepackage{array}
\usepackage{stfloats}
\usepackage[colorlinks,
            linkcolor=blue,       
            anchorcolor=blue,  
            citecolor=blue,        
            ]{hyperref}
\usepackage{caption}
\usepackage{graphicx}
\usepackage{epstopdf}
\theoremstyle{plain}
\newtheorem{thm}{Theorem}[section]
\newtheorem{cor}[thm]{Corollary}
\newtheorem{lem}[thm]{Lemma}
\newtheorem{prop}[thm]{Proposition}
\theoremstyle{definition}

\theoremstyle{plain}

\newtheorem{claim}{Claim}
\theoremstyle{plain}
\newtheorem{conj}{Conjecture}

\theoremstyle{plain}

\theoremstyle{plain}

\theoremstyle{plain}

\DeclareGraphicsExtensions{.eps,.eps.gz}
\usepackage[top=2.5cm,bottom=2.5cm,left=2.8cm,right=2.2cm]{geometry}
\normalsize \rm
\allowdisplaybreaks[4]
\makeatletter
\@addtoreset{equation}{section}
\makeatother

 \usepackage{indentfirst}
\begin{document}
\begin{CJK}{GBK}{song}
\newcommand{\song}{\CJKfamily{song}}    
\newcommand{\fs}{\CJKfamiwly{fs}}        
\newcommand{\kai}{\CJKfamily{kai}}      
\newcommand{\hei}{\CJKfamily{hei}}      
\newcommand{\li}{\CJKfamily{li}}        

\begin{center}
{{\huge Solution to a conjecture on resistance distances of block tower graphs }} \\[18pt]
{\Large Wensheng Sun$^{1}$, Yujun Yang$^{2}$, Wuxian Chen$^{1}$,\footnotetext{*Corresponding author at E-mail address: shjxu@lzu.edu.cn } Shou-Jun Xu$^{1}$* }\\[6pt]
{ \footnotesize  $^{1}$School of Mathematics and Statistics, Gansu Center for Applied Mathematics, Lanzhou University, Lanzhou 730000 China\\
$^{2}$School of Mathematics and Information Sciences, Yantai University, Yantai 264005 China}
\end{center}
\vspace{1mm}
\begin{abstract}
Let $G$ be a connected graph. The resistance distance between two vertices $u$ and $v$ of $G$, denoted by $R_{G}[u,v]$, is defined as the net effective resistance between them in the electric network constructed from $G$ by replacing each edge with a unit resistor. The resistance diameter of $G$, denoted by $D_{r}(G)$, is defined as the maximum resistance distance among all pairs of vertices of $G$. Let $P_n=a_1a_2\ldots a_n$ be the $n$-vertex path graph and $C_{4}=b_{1}b_2b_3b_4b_{1}$ be the 4-cycle. Then the $n$-th block tower graph $G_n$ is defined as the the Cartesian product of $P_n$ and $C_4$, that is, $G_n=P_{n}\square C_4$. Clearly, the vertex set of $G_n$ is $\{(a_i,b_j)|i=1,\ldots,n;j=1,\ldots,4\}$. In [Discrete Appl. Math. 320 (2022) 387--407], Evans and Francis proposed the following conjecture on resistance distances of $G_n$ and $G_{n+1}$:
\begin{equation*}
\lim_{n \rightarrow \infty}\left(R_{G_{n+1}}[(a_{1},b_1),(a_{n+1},b_3)]-R_{G_{n}}[(a_{1},b_1),(a_{n},b_3)]\right)=\frac{1}{4}.
\end{equation*}

In this paper, combining algebraic methods and electrical network approaches, we confirm and further generalize this conjecture. In addition, we determine all the resistance diametrical pairs in $G_n$, which enables us to give an equivalent explanation of the conjecture.

\vspace{1mm}
\noindent {\bf Keywords:} Cartesian product; Resistance distance; Resistance diameter; Asymptotic property; Principle of substitution
\end{abstract}

\section{Introduction}\label{section1}
As an important intrinsic graph metric, the distance functions on graphs have always attracted wide-spread attention and research. The most fundamental and widely used distance function defined on graphs is the shortest-path distance, or simple distance, where the \emph{distance} between two vertices $u$ and $v$, denoted by $d_G[u,v]$, is defined as the length of a shortest-path connecting $u$ and $v$ in $G$. Three decades ago, Klein and Randi\'{c} \cite{djk4} proposed another novel distance function which originates from electrical networks, named resistance distance. The \emph{resistance distance} between two vertices $u$ and $v$ of $G$, denoted by $R_{G}[u,v]$, is defined to be the effective resistance among nodes $u$ and $v$ in the corresponding electrical network (based on Ohm's and Kirchhoff's laws) constructed from $G$ by replacing each edge with a unit resistor. Lov\'{a}sz and Vesztergombi showed that the resistance distance is a harmonic function defined on graphs \cite{llo}. For the tree graph $T$,  the resistance distance between any two vertices is equivalent to the distance between them \cite{djk4}, i.e. $R_{T}[u,v]=d_T[u,v]$ for each pair of vertices $u,v\in V(T)$. From this perspective, resistance distance extends the concept of simple distance in some certain extent. More importantly, resistance distance seems to be more suitable than (shortest-path) distance to analyze the wave- or fluid-like communication among chemical molecules or  in the complex networks. For this reason, resistance distances have been widely studied in mathematical, physical and chemical literatures, see recent papers \cite{aaz2,ysun2,ppm2,agh,eje,wba,kde,jli,wsa,ysh,jge} and references therein.

The definition of shortest path distance on graphs, which leads to an elementary and widely studied parameter: diameter.   The \emph{diameter} of $G$, denoted by $D(G)$, is defined as the maximum distance between all pair of vertices in $G$. Inspired by the definition of diameter,  some scholars began to consider the \emph{resistance diameter} $D_{r}(G)$ of $G$, i.e. the maximum resistance distances between all pairs of vertices in $G$. We call vertex pair $u, v \in V(G)$ a \emph{resistance diametrical pair} if $R_{G}[u,v]=D_{r}(G)$, and use $RD(G)$ to denote the set of all resistance diametrical pairs in $G$. In 2020, Sardar et al. \cite{mss} took lead in studying the resistance diameter and gave an explicit combinatorial formula for calculating resistance diameter on hypercubes. Then, Li et al. \cite{yli} determined  resistance diameters of the two product operations(Cartesian and lexicographic product) of paths.  Motivated by the relationship between the diameter of a graph and that of its line graph given in \cite{kns}, Xu et al. \cite{sxu} also studied the relationship between the resistance diameter of graph and its line graph. Later, Sun and Yang \cite{wsu} confirmed the conjecture on resistance diameter of lexicographic product of paths which proposed in \cite{yli}. More recently, Vaskouski and Zadarazhniuk \cite{mva} studied the resistance diameters of cayley graphs on irreducible complex reflection groups.

\begin{figure*}[h]
  \setlength{\abovecaptionskip}{-0.1cm} %
  \centering
 $$\includegraphics[width=4.5in]{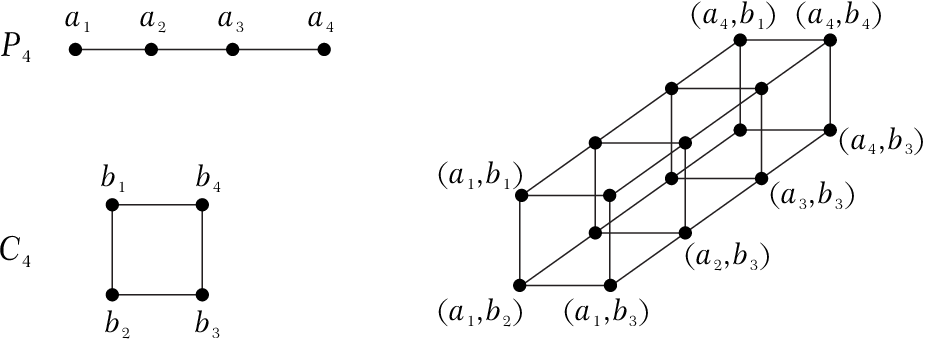}$$\\
 \caption{The block tower graph $G_4=P_{4}\square C_4$ and its vertex labelling.}     \label{Fig.1}
\end{figure*}

In this paper, we devote ourselves to a conjecture concerning resistance distances of the Cartesian product of path $P_n$ and cycle $C_4$. For a connected graph $G$, we use $u\sim v$ to denote $u$ and $v$ are adjacent in $G$. The \emph{Cartesian product} (or \emph{simple product}) \cite{sab} of two graphs $G$ and $H$, denoted by $G\square H$, is the graph with vertex set $V(G)\times V(H)$, and $(u,x)\sim (v,y)$ in $G\square H$ if and only if $u=v$ and  $x\sim y$ in $H$, or $x=y$ and $u\sim v$ in $G$. Let $P_n=a_1a_2\ldots a_n$ be the $n$-vertex path graph  and $C_{4}=b_{1}b_2b_3b_4b_{1}$ be the 4-cycle. Then the $n$-th block tower graph $G_n$ is defined as the Cartesian product of $P_n$ and $C_4$ with vertex set $\{(a_i,b_j)|i=1,\ldots,n;j=1,\ldots,4\}$, that is, $G_n=P_{n}\square C_4$. As an illustrative example, the block tower graph $G_4$ is given in Figure \ref{Fig.1}. Recently, Evans and Francis \cite{eje} gave a summary survey of methods used to calculate resistance distance in structured graphs. In addition, they proposed the following conjecture on an asymptotic formula for the resistance distances of $G_n$ and $G_{n+1}$: (see Conjecture 4.6 in \cite{eje})

\begin{conj}\label{conj1} Let $G_n=P_{n}\square C_4$. Then
$$\lim_{n \rightarrow \infty}\big(R_{G_{n+1}}[(a_{1},b_1),(a_{n+1},b_3)]-R_{G_{n}}[(a_{1},b_1),(a_{n},b_3)]\big)=\frac{1}{4}.$$
\end{conj}
It should be mentioned that only using the Kirchhoff matrix or Kirchhoff's current law to solve this conjecture would become overly complicated, due to the difficult-to-handle relationship between eigenvalues (or eigenvectors) in resistance distance formulas of $G_n$ and $G_{n+1}$. To this end, we combine electrical network methods and algebraic methods, then  confirm a generalized version of this conjecture (see Theorem \ref{Tm4.1}). In addition, we also determine all the resistance diametrical pairs in $G_n$, which enables us to give an equivalent explanation for this conjecture from the direction on resistance diameter.
\section{Preliminaries} \label{section2}

In this section, we first give some necessary definitions and useful transformations in electrical network theory. Since resistance of a resistor on network can always be viewed as a weighted edge of its corresponding graph, we do not distinguish between electrical networks and corresponding graphs. For convenience, if network $N$ is clear from the context, we use the notation $r_{N}[u, v] := r[u, v]$ for the resistance of the resistor on $e=[u,v]$.
Now we introduce some important principles in electrical networks. The following series and parallel principles are the most commonly used tools.

\textbf{Series principle}. If there are $n$ resistors with resistances of $r_{1}$, $r_{2}$,..., $r_{n}$ in series between two nodes $a$ and $b$, then we can use a single resistor with resistance $r_{ab}$ to replace those resistors, where $r_{ab} = r_{1} + r_{2} + \cdots + r_{n}$.

\textbf{Parallel principle}. If there are $n$ resistors with resistances of $r_{1}$, $r_{2}$,..., $r_{n}$ in parallels between two nodes $a$ and $b$, then we can use a single resistor with resistance $r_{ab}$ to replace those resistors, where $r_{ab} = (\frac{1}{ r_{1}} + \frac{1}{ r_{2}} + \cdots + \frac{1}{ r_{n}})^{-1}$.

For a connected graph $G$, a \textit{cut vertex} is a vertex such that its deletion results in a disconnected graph. Then $G$ is called \emph{nonseparable} if it is connected, nontrivial and contains no cut vertex \cite{dbw}.
Then a \emph{block} in  $G$ is a maximal nonseparable subgraph. For graphs with a cut vertex, the following principle can simplify the calculation of resistance distances.

\textbf{Principle of elimination}\cite{djk4}. Let $N$ be a connected network and $B$ be a block of $N$ containing exactly one cut vertex $x$ of $N$. If $H$ is the network obtained from $N$ by deleting all the vertices of $B$ except $x$, then for $u,v\in V(H)$, $R_{H}[u, v] = R_{N}[u, v]$.

Suppose that $N$ and $M$ are two networks and vertex set $S\subseteq V(N)\cap V(M)$.  Then $N$ and $M$ are said to be \textit{$S$-equivalent} if $R_{N}[u, v]= R_{M}[u, v]$ for all $u,v\in S$.

\textbf{Principle of substitution}. Let $H$ be a subnetwork of $N$. If $H$ is $V(H)$-equivalent to $H^*$, then the network $N^*$ obtained from $N$ by replacing $H$ with $H^*$
satisfies $R_N(u,v)=R_{N^*}(u,v)$ for all $u,v\in V(N)$.

When we perform equivalent network substitutions, a well-known common  used tool is the \textit{$\Delta-Y$ transformation}\cite{aek}. The conversion process is shown in Figure \ref{Fig.2.}. If the resistances of resistors in $\Delta$ and $Y$-networks satisfy
\begin{equation}\label{eq2.1}
r_a=\frac{r_1r_3}{r_1+r_2+r_3},\quad r_b=\frac{r_1r_2}{r_1+r_2+r_3},\quad r_c=\frac{r_2r_3}{r_1+r_2+r_3}.
\end{equation}
\begin{figure*}[ht]
  \setlength{\abovecaptionskip}{-0.1cm} %
  \centering
 $$\includegraphics[width=3.2in]{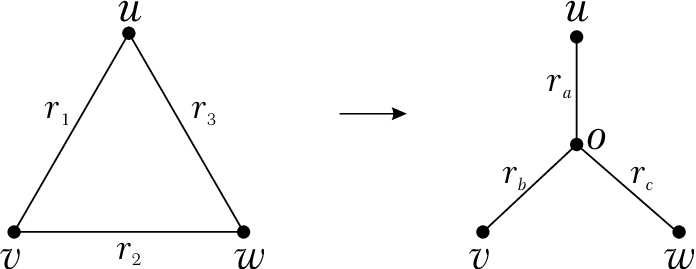}$$\\
 \caption{$\Delta-Y$ transformation.}     \label{Fig.2.}
\end{figure*}
Then $\Delta$ and $Y$-networks are $\{u,v,w\}$-equivalent.

Let $K_{m,n}$ be a complete bipartite graph with partial sets $V(\overline{K_m})=\{x_1,x_2,\cdots,x_m\}$ and $V(\overline{K_n})=\{y_1,y_2,\cdots,y_n\}$. Using Forster's first theorem, Klein \cite{djk2} gave the resistance distances between any two vertices in $K_{m,n}$.
\begin{thm}\cite{djk2}\label{tm2.1}
For a complete bipartite network graph $K_{m,n}$, let $u$ and $v$ be distinct nodes in $K_{m,n}$. Then
 \end{thm}
$$ R_{K_{m,n}}[u,v]=\left\{
\begin{array}{ll}
 \frac{2}{n}         & \mbox{if} \ u, v \in V(\overline{K_m}),  \\

 \frac{2}{m}         & \mbox{if} \ u, v \in V(\overline{K_n}),  \\

 \frac{m+n-1}{mn}    & \mbox{if} \ u \in V(\overline{K_m}) \ \mbox{and}\  v \in V(\overline{K_n}).

\end{array}
\right.$$

In \cite{svg}, Gervacio defined the network $K^*_{m, n}$ (see Figure \ref{Fig.8}) obtained from  $K_{m, n}$ as a network which consists of  $m + n + 2$ nodes $ \{x_0, x_1, ..., x_m, y_0, y_1,..., y_n\}$, and $m+n+1$ edges $\{[x_0, y_0], [x_0, x_1],\ldots,$ $[x_0, x_m], [y_0, y_1], \ldots, [y_0, y_n]\}$, where the resistances on edges $[x_0, x_i]$ ($1\leq i\leq m$) are $\frac{1}{n}$, the resistances on edges $[y_0, y_j]$ ($1\leq j\leq n$) are $\frac{1}{m}$, and the resistance on edge $[x_0, y_0]$ is $-\frac{1}{mn}$. According to principle of substitution,  Gervacio obtained the following result.
\begin{figure*}[h]
  \setlength{\abovecaptionskip}{0cm} %
  \centering
 $$\includegraphics[width=3.6in]{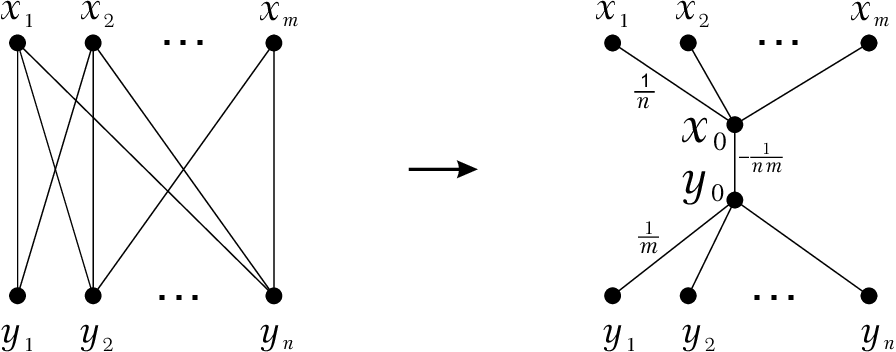}$$\\
 \caption{Networks $K_{m,n}$ and $K^{*}_{m,n}$.}     \label{Fig.8}
\end{figure*}

\begin{thm}\cite{svg}\label{tm2.2}
The network $K_{m,n}$ is $V(K_{m,n})$-equivalent to network $K^*_{m, n}$.
\end{thm}

\begin{lem}\label{lem2.33}
 \textbf{Rayleigh's Monotonicity Law}\cite{jst}. If the edge-resistances in an electrical network are increased, then the effective resistance between any two nodes in the network does not decrease.
\end{lem}
For a network $N$, if an edge (resp. node) in $N$ is deleted, it is equivalent to increase the resistance on this edge (resp. all edges that incident with this node) to $+ \infty$. Therefore, according to Rayleigh's monotonicity law, the following proposition can be readily known.

\begin{prop}\label{prop2.4}
Let $M$ be a subnetwork of $N$. Then for any $u,v \in V(M)$
\begin{center}$R_{N}[u, v] \leq R_{M}[u, v]$. \end{center}
\end{prop}

In the rest of this section, we give some formulas for calculating resistance distances through algebraic methods. First, the \textit{adjacency matrix} $A_G$ of $G$ with $n$ vertices is an $n \times n$ matrix such that  the $(i, j)$-th element of $A_G$ is equal to 1 if vertices $i$ and $j$ are adjacent and 0 otherwise. The \textit{Laplacian matrix} of $G$ is $L_G=D_G-A_G$, where $D_G$ is the diagonal matrix of vertex degrees of $G$. Obviously, if $G$ is a connected graph, then $L_{G}$ has a single $0$-eigenvalue, and all other eigenvalues are positive.

Given two graphs $G$ and $H$ with $n$ vertices and $m$ vertices, respectively. In the following, we suppose that $\lambda_{1}, \lambda_{2},...,\lambda_{n}=0$ and $\mu_{1}, \mu_{2},...,\mu_{m}=0$ are eigenvalues of $L_{G}$ and $L_{H}$ arranged in nonincreasing order. Suppose that $\Psi_{1}, \Psi_{2},...,\Psi_{n}$ and $\Phi_{1}, \Phi_{2},...,\Phi_{m}$ denote orthogonal eigenvectors of $L_{G}$ and $L_{H}$ such that for any $i$ and $j$, $\Psi_{i}=(\Psi_{i1},\Psi_{i2},...,\Psi_{in})^T$ and  $\Phi_{j}=(\Phi_{j1},\Phi_{j2},...,\Phi_{jm})^T$ are eigenvectors affording $\lambda_{i}$ and $\mu_{j}$, respectively. It is understood that  $\Psi_{n} = \frac{1}{\sqrt{n}}\textbf{e}_{n}$ and $\Phi_{m} = \frac{1}{\sqrt{m}}\textbf{e}_{m}$, where $\textbf{e}_{k}$ denotes the $k$-vector of all 1 s.

For two vectors $\Psi$  and $\Phi$, we use $\Psi \bm{\otimes} \Phi $ to denote their Kronecker product.  Then Laplacian eigenvalues and eigenvectors of $G \square H$ are given in the following lemma.

\begin{lem}\cite{rme}\label{lem2.6}
Let two graphs $G$ and $H$ be defined as above. Then the eigenvalues of $G\square H$ are $\{\lambda_{i}+\mu_{j}: 1\leq i\leq n, 1\leq j \leq m\}$. Moreover, the eigenvector of $G\square H$ affording $\lambda_{i}+\mu_{j}$ is $\Psi_{i} \bm{\otimes} \Phi_{j}$.
\end{lem}
In \cite{yya1}, Yang and Klein further verified the eigenvectors of $G\square H$ given above are orthogonal (see Theorem 2.5), they also characterized the explicit formula for resistance distances in graph $G\square H$, which plays an essential tool in the proof, as given in the following theorem.

\begin{thm}\cite{yya1}\label{thm3.1} Let two graphs $G$ and $H$ be defined as above. For $u, v\in V(G)$ and $x,y \in V(H)$, we have
\begin{equation*}
R_{G\square H}[(u,x),(v,y)]=\frac{1}{m}R_{G}[u,v] +  \frac{1}{n}R_{H}[x,y] + \sum_{p=1}^{n-1}\sum_{q=1}^{m-1}\frac{\left(\Psi_{pu}\Phi_{qx}-\Psi_{pv}\Phi_{qy}\right)^2}{\lambda_{p}+\mu_{q}}.
\end{equation*}
\end{thm}
The join $G+H$ \cite{aaz1} is a graph such that $V(G+H)=V(G)\cup V(H)$ and $E(G+H)=E(G)\cup E(H)\cup \{[u,v]|u\in V(G), v\in V(H)\}$.
\begin{lem}\cite{yya1}\label{tm1.2}
Let two graphs $G$ and $H$ be defined as above. Then for $u, v\in V(G)$ and $w \in V(H)$, we have
\begin{equation*}
R_{G+H}[u,v]=\sum_{k=1}^{n-1}\frac{(\Psi_{ku}-\Psi_{kv})^2}{\lambda_{k}+m};
\end{equation*}
\begin{equation*}
R_{G+H}[u,w]=\sum_{k=1}^{n-1}\frac{\Psi_{ku}^2}{\lambda_{k}+m}+ \sum_{k=1}^{m-1}\frac{\Phi_{kw}^2}{\mu_{k}+n}+\frac{1}{nm}.
\end{equation*}
\end{lem}

For a clique $K_{2}=c_{1}c_{2}$,  the $n$-th ladder graph $L_n$ is the Cartesian product of $P_n$ and $K_2$, that is $L_n=P_{n}\square K_2$.  For example, the ladder graph $L_4$ is given in Figure \ref{Fig.7}. Z. Cinkir \cite{zci} and Shi et al. \cite{lsh} independently obtained formulae for resistance distance between any two vertices in $L_n$. Z. Cinkir \cite{zci} also determined the Kirchhoff indices (the sum of resistance distance between all pairs of vertices).

\begin{figure*}[h]
  \setlength{\abovecaptionskip}{-0.1cm} %
  \centering
 $$\includegraphics[width=3.2in]{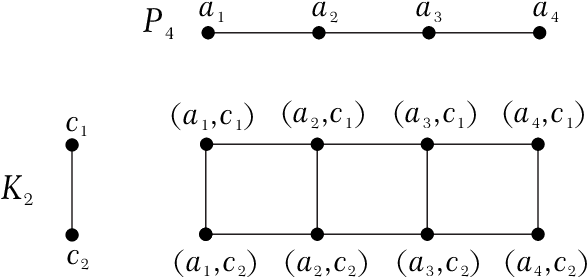}$$\\
 \caption{The ladder graph $L_4=P_{4}\square K_2$ and its vertex labelling.}     \label{Fig.7}
\end{figure*}

 \begin{lem}\cite{zci} \label{lem3.2}
 Let $L_n=P_{n}\square K_2$ $(n\geq 2)$. Then
$$R_{L_{n}}[(a_1,c_1),(a_{n},c_2)]=\frac{n-1}{2}+\frac{(1+\alpha^{n-1})}{4\sqrt{3}(1-\alpha^{2n})}\left(2+2\alpha^{n+1}+2\alpha^{n}+2\alpha\right),$$
where $\alpha=2-\sqrt{3}.$
\end{lem}
By using Lemma \ref{lem3.2}, together with the aid of Mathematica \cite{wol}, the following properties of resistance distances in $L_{n}$ can be derived straightforwardly.
 \begin{lem} \label{lem3.3}
Let $L_n=P_{n}\square K_2$ $(n\geq 2)$. Then
$$R_{L_{n+1}}[(a_1,c_1),(a_{n+1},c_2)]-R_{L_{n}}[(a_1,c_1),(a_{n},c_2)]> \frac{1}{4}.$$
\end{lem}

\begin{lem}\cite{lsh} \label{lem3.4}
Let $L_n=P_{n}\square K_2$ $(n\geq 2)$. Then
$$R_{L_{n}}[(a_1,c_1),(a_{n},c_2)]-R_{L_{n}}[(a_1,c_1),(a_{n},c_1)]=\frac{2\sqrt{3}}{a^{n}-b^{n}}> 0,$$
where $a=2+\sqrt{3}$, $b=2-\sqrt{3}$.
\end{lem}

\section{Resistance distance properties in weighted cone graph}
First, we define a weighted cone graph $C^{m}_{G}$ on graph $G$. Given an isolated vertex $b$, then the cone graph \cite{caa} is obtained by connecting vertex $b$ to each vertex in $G$. For weighted cone graph $C^{m}_{G}$, all edges belong to $E(G)$ have same weight 1, and all edges $\{[u,b]|u\in V(G)\}$ have same weight $\frac{1}{m}$, where $m> 1$ is an integer. Then we have the following result.
 \begin{lem}\label{lem3.6}
Let $C^{m}_{G}$ be a weighted cone graph. Then for distinct vertices $u,v \in V(G)$
\begin{equation*}
R_{C^{m}_{G}}[u,v]=\sum_{k=1}^{n-1}\frac{(\Psi_{ku}-\Psi_{kv})^2}{\lambda_{i}+m},
\end{equation*}
\begin{equation*}
R_{C^{m}_{G}}[u,b]=\sum_{k=1}^{n-1}\frac{\Psi_{ku}^2}{\lambda_{k}+m}+\frac{1}{nm}.
\end{equation*}
where graph $G$ uses the notation of the section \ref{section2}.
\end{lem}
\begin{figure*}[h]
  \setlength{\abovecaptionskip}{0cm} %
  \centering
 $$\includegraphics[width=3.6in]{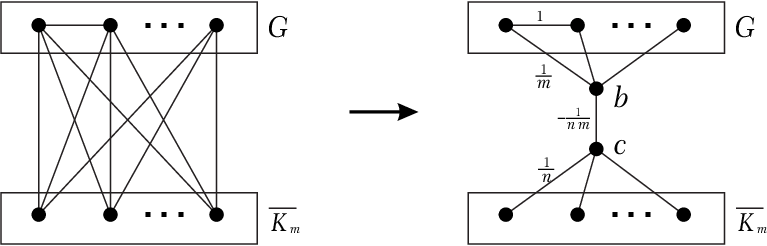}$$\\
 \caption{Illustration of network simplification for graph $G+\overline{K_{m}}$.}     \label{Fig.9}
\end{figure*}
\begin{proof}
In order to obtain the resistance distance in graph $C^{m}_{G}$,  we consider the join graph $G+\overline{K_{m}}$. Since the edges between $V(G)$ and $V(\overline{K_{m}})$ form a complete bipartite graph $K_{n,m}$, by using Theorem \ref{tm2.2}, we substitute complete bipartite graphs $K_{n,m}$ with $K^{*}_{n,m}$, then obtain the equivalent network $N^{*}$ of $G+\overline{K_{m}}$ (see Figure \ref{Fig.9}), and weights on edges satisfy as\\
$r[u,v]=1$; $r[b,c]=-\frac{1}{nm}$; $r[u,b]=\frac{1}{m}$; $r[c,w]=\frac{1}{n}$,\\
where $u,v \in V(G)$ and $w\in V(\overline{K_{m}})$.\\
Observe that vertex set $V(G)\cup \{b\}$ in $N^{*}$ exactly induced the graph $C^{m}_{G}$. By principles of substitution and elimination, for $u, v\in V(G)$, we have
\begin{equation*}
R_{C^{m}_{G}}[u,v]=R_{N^{*}}[u,v]=R_{G+\overline{K_{m}}}[u,v]=\sum_{k=1}^{n-1}\frac{(\Psi_{ku}-\Psi_{kv})^2}{\lambda_{k}+m}.
\end{equation*}
On the other hand, for $u\in V(G)$ and $w \in V(\overline{K_{m}})$,  we have
\begin{align}\label{eq1.5}
R_{G+\overline{K_{m}}}[u,w]&= R_{N^{*}}[u,w]\notag\\
 &=R_{N^{*}}[u,b]+R_{N^{*}}[b,c]+R_{N^{*}}[c,w] \nonumber \\
              &=R_{C^{m}_{G}}[u,b]+\frac{1}{n}-\frac{1}{nm}.
\end{align}
Since $\mu_{k}=0$ for $1\leq k \leq m$ and $\sum_{k=1}^{m-1}\Phi_{kw}^{2}=1-\frac{1}{m}$, by Lemma \ref{tm1.2},  we have
\begin{align}\label{eq3.6}
R_{G+\overline{K_{m}}}[u,w]&=\sum_{k=1}^{n-1}\frac{\Psi_{ku}^2}{\lambda_{k}+m}+ \sum_{k=1}^{m-1}\frac{\Phi_{kw}^2}{\mu_{k}+n}+\frac{1}{nm} \notag\\
              &= \sum_{k=1}^{n-1}\frac{\Psi_{ku}^2}{\lambda_{k}+m} +\frac{1-\frac{1}{m}}{n}+\frac{1}{nm}  \nonumber \\
              &=\sum_{k=1}^{n-1}\frac{\Psi_{ku}^2}{\lambda_{k}+m}+ \frac{1}{n}.
\end{align}
Combining Eqs. (\ref{eq1.5}) and (\ref{eq3.6}), we conclude
$$R_{C^{m}_{G}}[u,b]=\sum_{k=1}^{n-1}\frac{\Psi_{ku}^2}{\lambda_{k}+m}+\frac{1}{nm}.$$
This completes the proof.
\end{proof}

In the following, we devote ourselves to weighted cone graphs $C^{m}_{P_n}$ for $n$-vertex path graph $P_n=a_1a_2\ldots a_n$. In general, $C^{m}_{P_n}$ is also called a weighted fan graph.  As an illustrative example, graph $C^{m}_{P_{n+1}}$ shown in Figure \ref{Fig.3} (a). In \cite{wsu}, the authors in this paper studied some resistance distance monotonicity properties on $C^{m}_{P_n}$. Here we give some resistance distance asymptotic properties for networks $C^{m}_{P_n}$ and $C^{m}_{P_{n+1}}$ by using the electrical network method. These results play essential rules in the proof of main results.

\begin{lem} \label{lem3.1}
For weighted cone graphs $C^{m}_{P_n}$ and $C^{m}_{P_{n+1}}$, we have
\begin{equation*}0 < R_{C^{m}_{P_{n+1}}}[a_1,a_{n+1}]-R_{C^{m}_{P_{n}}}[a_1,a_n] < \frac{2}{m^{n}}.\end{equation*}
\end{lem}
\begin{proof}
\begin{figure*}[h]
  \setlength{\abovecaptionskip}{0cm} %
  \centering
 $$\includegraphics[width=6.0in]{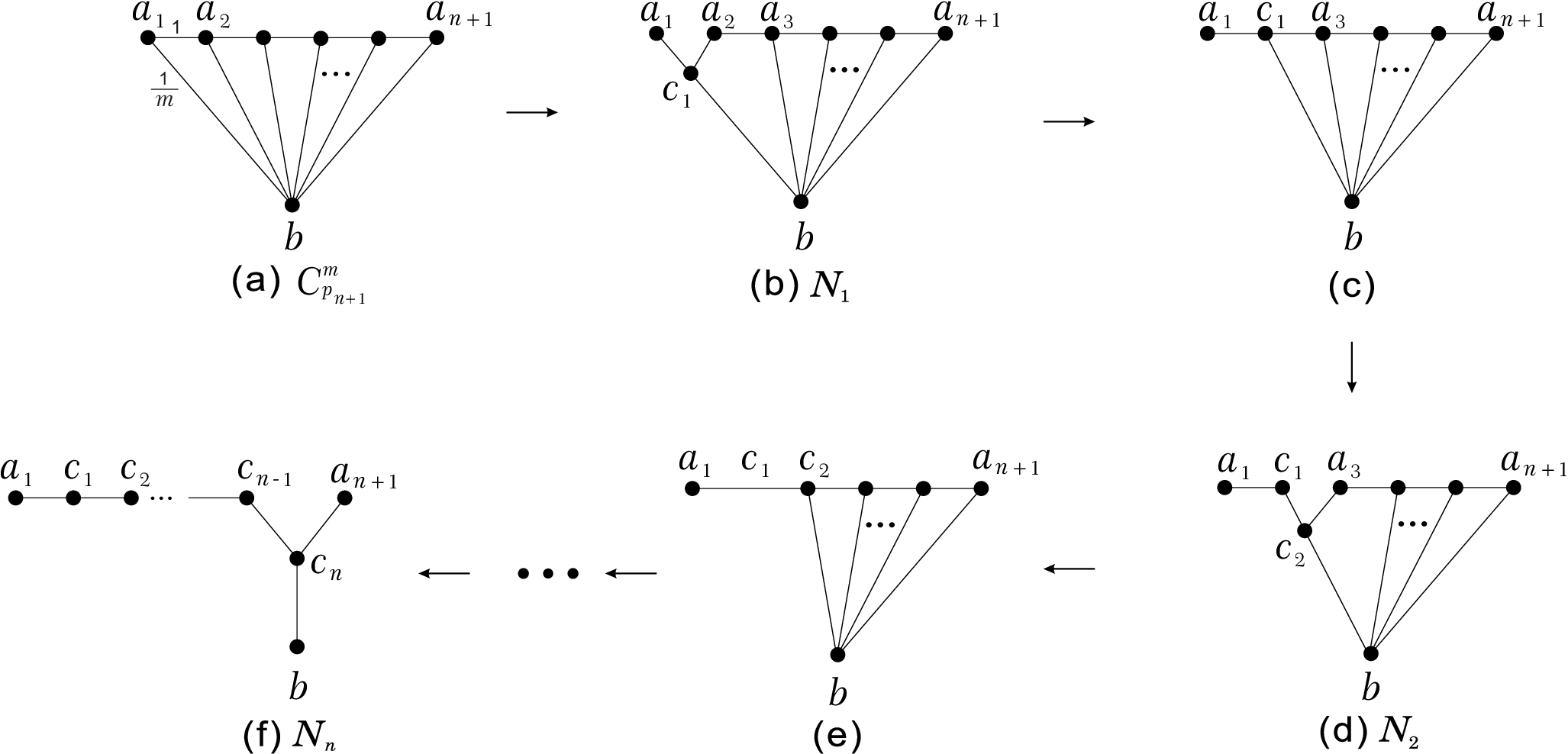}$$\\
 \caption{Illustration of network simplifications in the proof of Lemma \ref{lem3.1}.}    \label{Fig.3}
\end{figure*}
In order to compare $R_{C^{m}_{P_{n+1}}}[a_1,a_{n+1}]$ and $R_{C^{m}_{P_{n}}}[a_1,a_n]$, we first make network simplifications on network $C^{m}_{P_{n+1}}$. Since vertex subset $\{a_1,a_2,b\}$ form a $\Delta$-network, using $\Delta-Y$ transformation, we replace the triangle $a_1a_2b$ in network $C^{m}_{P_{n+1}}$ with a $Y$-network with $c_{1}$ as center, we obtain the simplified network $N_{1}$, as shown in Figure \ref{Fig.3} (b). Then according to the series principle, we replace the path $c_{1}a_2a_3$ with a new weighted edge $c_{1}a_3$ and the network is obtained (see Figure \ref{Fig.3} (c)). Similarly, continue to make  $\Delta-Y$ transformation to the network in Figure \ref{Fig.3} (c) by replacing the triangle $c_{1}a_3 b$ with a $Y$-network with $c_{2}$ as center. Then we obtain the network $N_{2}$, see Figure \ref{Fig.3} (d). Repeat the above operations until no triangle network exists, we could eventually obtain the simplified network $N_{n}$ as shown in Figure \ref{Fig.3} (f). By the series principle, we have
\begin{equation} \label{equ3.1} R_{C^{m}_{P_{n+1}}}[a_1,a_{n+1}] = R_{N_{n}}[a_1,c_{n}]+R_{N_{n}}[c_{n},a_{n+1}].\end{equation}
Similarly, we make the same network simplifications to network $C^{m}_{P_{n}}$ as above, then obtain the simplified network $M_{n-1}$ as shown in Figure \ref{Fig.4}, we have
\begin{figure*}[h]
  \setlength{\abovecaptionskip}{-0.1cm} %
  \centering
 $$\includegraphics[width=1.5in]{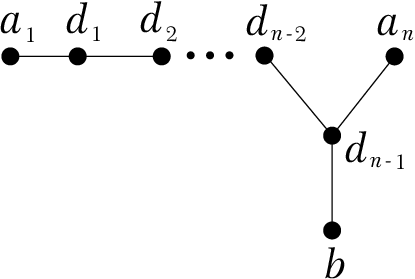}$$\\
 \caption{Simplified network $M_{n-1}$.}     \label{Fig.4}
\end{figure*}
\begin{equation} \label{equ3.2} R_{C^{m}_{P_{n}}}[a_1,a_{n}] = R_{M_{n-1}}[a_1,d_{n-1}]+R_{M_{n-1}}[d_{n-1},a_{n}].\end{equation}
According to the symmetry of network $C^{m}_{P_{n+1}}$ and $C^{m}_{P_{n}}$, we see
\begin{equation*} R_{C^{m}_{P_{n+1}}}[a_1,b]=R_{C^{m}_{P_{n+1}}}[a_{n+1},b]; \ R_{C^{m}_{P_{n}}}[a_1,b]=R_{C^{m}_{P_{n}}}[a_n,b].\end{equation*}
This means
\begin{equation}\label{equ3.3}R_{N_{n}}[a_1,c_{n}]=R_{N_{n}}[c_{n},a_{n+1}]; \ R_{M_{n-1}}[a_1,d_{n-1}]=R_{M_{n-1}}[d_{n-1},a_n].\end{equation}
Then combining Eqs. (\ref{equ3.1})-(\ref{equ3.3}), we have
\begin{equation} \label{equ3.4} R_{C^{m}_{P_{n+1}}}[a_1,a_{n+1}]-R_{C^{m}_{P_{n}}}[a_1,a_n]=2(R_{N_{n}}[a_1,c_{n}]-R_{M_{n-1}}[a_1,d_{n-1}]).\end{equation}
Note that in the same network simplification for networks $C^{m}_{P_{n+1}}$ and $C^{m}_{P_{n}}$, we have
\begin{equation} \label{equ3.5} R_{N_{n}}[a_1,c_{n-1}]=R_{M_{n-1}}[a_1,d_{n-1}];\  R_{N_{n}}[a_1,c_{n}]=R_{N_{n}}[a_1,c_{n-1}]+R_{N_{n}}[c_{n-1},c_{n}].\end{equation}
Substituting Eq. (\ref{equ3.5}) into Eq. (\ref{equ3.4}), we know that
\begin{equation}\label{equ3.6} R_{C^{m}_{P_{n+1}}}[a_1,a_{n+1}]-R_{C^{m}_{P_{n}}}[a_1,a_n]=2R_{N_{n}}[c_{n-1},c_{n}].\end{equation}
Therefore, we have
$$R_{C^{m}_{P_{n+1}}}[a_1,a_{n+1}]-R_{C^{m}_{P_{n}}}[a_1,a_n]>0.$$
On the other hand, in the each network simplification process from network $N_{i}$  to $N_{i+1}$, where $1\leq i \leq n-1$, see Figure \ref{Fig.5}. Since in $\Delta$-network $c_{i}a_{i+2}b$, the weight on the edge $c_{i}a_{i+2}$ is great than $1$, then using Eq. (\ref{eq2.1}), we have
\begin{figure*}[h]
  \setlength{\abovecaptionskip}{-0.1cm} %
  \centering
 $$\includegraphics[width=4.0in]{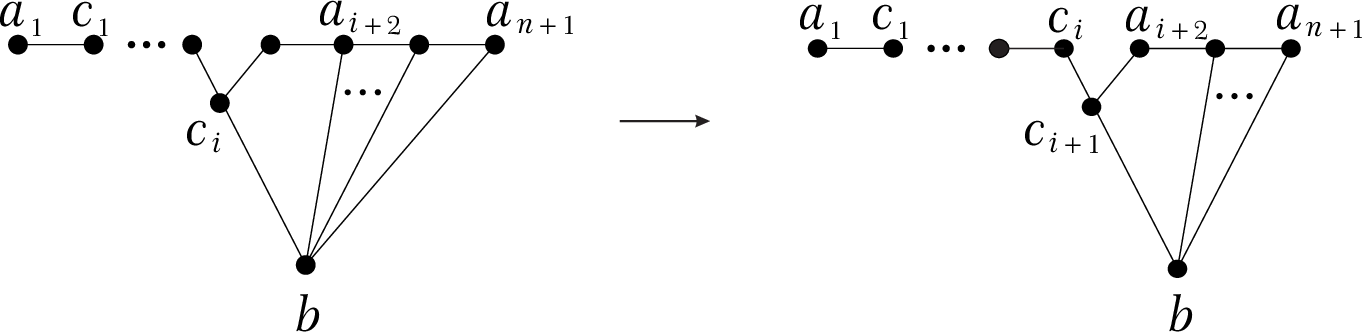}$$\\
 \caption{Illustration of network simplification from $N_{i}$ to $N_{i+1}$.}     \label{Fig.5}
\end{figure*}
$$r_{N_{1}}[c_{1},b]<\frac{1}{m^{2}},\ r_{N_{i+1}}[c_{i+1},b]<\frac{r_{N_{i}}[c_{i},b]}{m}.$$
Hence we could conclude that
\begin{equation}\label{equ3.7}r_{N_{n-1}}[c_{n-1},b]<\frac{1}{m^{n}}.\end{equation}
Again by using Eq. (\ref{eq2.1}), we obtain
$$r_{N_{n}}[c_{n-1},c_{n}]< r_{N_{n-1}}[c_{n-1},b]<\frac{1}{m^{n}}.$$
Observe that $r_{N_{n}}[c_{n-1},c_{n}]=R_{N_{n}}[c_{n-1},c_{n}]$, then according to Eq. (\ref{equ3.6}), we have
$$R_{C^{m}_{P_{n+1}}}[a_1,a_{n+1}]-R_{C^{m}_{P_{n}}}[a_1,a_n]=2\cdot r_{N_{n}}[c_{n-1},c_{n}]<\frac{2}{m^{n}}.$$
This completes the proof.
 \end{proof}

\begin{lem} \label{lem3.111}
For weighted cone graphs $C^{m}_{P_n}$ and $C^{m}_{P_{n+1}}$, we have
\begin{equation*}0<R_{C^{m}_{P_n}}[a_1,b] -R_{C^{m}_{P_{n+1}}}[a_1,b]<\frac{1}{m^{n}}.\end{equation*}
\end{lem}
 \begin{proof}
In networks $N_n$ (resp. $M_{n-1}$) as shown in Figure \ref{Fig.3} (f) (resp. Figure \ref{Fig.4}). We have
 \begin{equation}\label{eq3.8}R_{C^{m}_{P_{n+1}}}[a_1,b] = R_{N_{n}}[a_1,c_{n}]+R_{N_{n}}[c_{n},b],\end{equation}
 \begin{equation}\label{eq3.9}R_{C^{m}_{P_n}}[a_1,b] = R_{M_{n-1}}[a_1,d_{n-1}]+R_{M_{n-1}}[d_{n-1},b].\end{equation}
By using Eqs. (\ref{equ3.5}), (\ref{eq3.8}) and (\ref{eq3.9}), we have
\begin{equation*}  R_{C^{m}_{P_n}}[a_1,b] -  R_{C^{m}_{P_{n+1}}}[a_1,b] < R_{M_{n-1}}[d_{n-1},b].\end{equation*}
In addition, as obtained Eq. (\ref{equ3.7}) in proof in Lemma \ref{lem3.1}, we know
 \begin{equation}\label{eq3.10}R_{M_{n-1}}[d_{n-1},b]=R_{N_{n-1}}[c_{n-1},b]=r_{N_{n-1}}[c_{n-1},b]<\frac{1}{m^{n}}. \end{equation}
Therefore, we have
$$R_{C^{m}_{P_n}}[a_1,b] -R_{C^{m}_{P_{n+1}}}[a_1,b]<\frac{1}{m^{n}}.$$
On the other hand, $C^{m}_{P_n}$ is a proper subnetwork of $C^{m}_{P_{n+1}}$. By the series and parallel principles, it is not hard to see that
$$R_{C^{m}_{P_n}}[a_1,b] -R_{C^{m}_{P_{n+1}}}[a_1,b]> 0.$$
The proof is complete.
\end{proof}
\begin{lem} \label{lem3.11}
For weighted cone graph $C^{m}_{P_n}$, we have
\begin{equation*} 0<2 R_{C^{m}_{P_{n}}}[a_1,b]-R_{C^{m}_{P_n}}[a_1,a_n] < \frac{2}{m^{n}}.\end{equation*}
\end{lem}

\begin{proof}
Similarly, in the network $M_{n-1}$ (see Figure \ref{Fig.4}), by the series principle, we have
 \begin{equation}\label{e3.2}R_{C^{m}_{P_n}}[a_1,b] = R_{M_{n-1}}[a_1,d_{n-1}]+R_{M_{n-1}}[d_{n-1},b],\end{equation}
 \begin{equation} \label{e3.3} R_{C^{m}_{P_n}}[a_1,a_n] = R_{M_{n-1}}[a_1,d_{n-1}]+R_{M_{n-1}}[d_{n-1},a_n].\end{equation}
Since $ R_{M_{n-1}}[a_1,d_{n-1}]= R_{M_{n-1}}[d_{n-1},a_n]$, we have
\begin{equation} \label{e3.4} R_{C^{m}_{P_n}}[a_1,a_n] = 2\cdot R_{M_{n-1}}[a_1,d_{n-1}].\end{equation}
Combining Eqs. (\ref{e3.2}) and (\ref{e3.4}),  then clearly
$$  2 R_{C^{m}_{P_n}}[a_1,b]-R_{C^{m}_{P_n}}[a_1,a_n]= 2R_{M_{n-1}}[d_{n-1},b]>0.$$
Together with Eq. (\ref{equ3.7}) in the proof of Lemma \ref{lem3.1}, we conclude that
$$2 R_{C^{m}_{P_n}}[a_1,b]-R_{C^{m}_{P_n}}[a_1,a_n] < \frac{2}{m^{n}}.$$
This completes the proof.
\end{proof}

\section{Main results} \label{section4}
In this section, we first prove a generalized version of Conjecture \ref{conj1} by using  electrical network approaches combined with  some asymptotic properties on resistance distances obtained in previous section. Without loss of generality, we always suppose that $\lambda_{1}, \lambda_{2},...,\lambda_{n}=0$ and $\mu_{1}, \mu_{2},...,\mu_{n}, \mu_{n+1}=0$ be Laplacian eigenvalues of path $P_{n}$ and $P_{n+1}$ arranged in nonincreasing order, respectively. Suppose that $\Psi_{1}, \Psi_{2},...,\Psi_{n}$ and $\Phi_{1}, \Phi_{2},...,\Phi_{n},\Phi_{n+1}$ denote orthogonal eigenvectors of $P_{n}$ and $P_{n+1}$, where $\Psi_{i}=(\Psi_{ia_{1}},\Psi_{ia_{2}},...,\Psi_{ia_{n}})^T$ and  $\Phi_{j}=(\Phi_{ja_{1}},\Phi_{ja_{2}},...,\Phi_{ja_{n+1}})^T$ are eigenvectors affording $\lambda_{i}$ and $\mu_{j}$, respectively.

In what follows, we introduce an important network, named hypercube \cite{fha}, which is defined recursively in terms of the Cartesian product as follows:
$$Q_1=K_2,\ Q_k=Q_{k-1}\square Q_1=\underbrace{K_2\square K_2 \square \cdots \square K_2}_{k}, \ k\geq 2.$$
As we all know, the class of hypercubes is the most common, multipurpose and efficient topological structure class of interconnection networks. Studying the various properties of hypercubes networks has always been a topic of great interest in the field of physics and mathematics, such as embeddings\cite{mli}, retracts\cite{hjb}, treewidth\cite{lsc}, matching extendability \cite{jva}, and so on.

In \cite{mss}, Sardar et al. gave an explicit combinatorial formula for calculating resistance diameter on hypercubes. Let $x$ and $y$ be two vertices at maximum distance in $Q_k$, they proved that
\begin{equation}\label{D}D_{r}(Q_k)=R_{Q_{k}}[x,y]= \sum_{i=1}^{k}\frac{(k-i)!(i-1)!}{k!}.\end{equation}
For convenience, let the vertex set of $Q_k$  be $V(Q_{k})=\{b_1, b_2,\cdots,b_{2^{k}}\}$. We also assume that $\nu_{1}, \nu_{2},...,\nu_{2^{k}}=0$ be Laplacian eigenvalues of hypercube $Q_{k}$ arranged in nonincreasing order. Suppose that $\Upsilon_{1}, \Upsilon_{2},...,\Upsilon_{2^{k}}$  denote orthogonal eigenvectors of $Q_{k}$, where $\Upsilon_{i}=(\Upsilon_{ib_{1}},\Upsilon_{ib_{2}},...,\Upsilon_{ib_{2^{k}}})^T$  be eigenvectors affording $\nu_{i}$.

Here we consider Cartesian product $U_{n}=P_{n}\square Q_{k}$. Obviously, for $k=2$, $U_{n}=P_{n}\square Q_{2}$ is the block tower graph $G_n$ studied in this paper. Then we give the main result.

\begin{thm}\label{Tm4.1}
Let $U_{n}=P_{n}\square Q_{k}$. Then for $i,j\in \{1,2,\cdots , 2^{k}\}$, we have
$$\lim_{n \rightarrow \infty}\big(R_{U_{n+1}}[(a_{1},b_i),(a_{n+1},b_j)]-R_{U_{n}}[(a_{1},b_i),(a_{n},b_j)]\big)=\frac{1}{2^{k}}.$$
\end{thm}
\begin{proof}
By utilizing Theorem \ref{thm3.1}, we have
\begin{align} \label{eq4.2}
R_{U_{n+1}}[(a_{1},b_i),&(a_{n+1},b_j)]-R_{U_{n}}[(a_{1},b_i),(a_{n},b_j)] \nonumber \\
&=\frac{1}{2^{k}}+\left(\frac{1}{n+1}R_{Q_{k}}[b_i,b_j]-\frac{1}{n}R_{Q_{k}}[b_i,b_j]\right)\nonumber \\
& + \sum_{q=1}^{2^{k}-1}\left(\sum_{p=1}^{n}\frac{\left(\Phi_{pa_{1}}\Upsilon_{qb_i}-\Phi_{pa_{n+1}}\Upsilon_{qb_j}\right)^2}{\mu_{p}+\nu_{q}}- \sum_{p=1}^{n-1}\frac{\left(\Psi_{pa_{1}}\Upsilon_{qb_i}-\Psi_
{pa_{n}}\Upsilon_{qb_j}\right)^2}{\lambda_{p}+\nu_{q}}\right).
\end{align}
For vertices $b_i$, $b_j \in V(Q_k)$, then Eq. (\ref{D}) implies
$$R_{Q_k}[b_i,b_j] \leq D_r(Q_k) \leq 1.$$
Thus it follows that
\begin{equation}\label{E}
\lim_{n \rightarrow \infty}\left(\frac{1}{n+1}R_{Q_{k}}[b_i,b_j]-\frac{1}{n}R_{Q_{k}}[b_i,b_j]\right)=0.
\end{equation}
By Lemma \ref{lem3.6}, for integer $m> 1$ we have
\begin{equation}\label{eq20}
R_{C^{m}_{P_n}}[a_{1},a_{n}]=\sum_{p=1}^{n-1}\frac{\left(\Psi_{pa_{1}}-\Psi_{pa_{n}}\right)^2}{\lambda_{p}+m},\  R_{C^{m}_{P_{n+1}}}[a_{1},a_{n+1}]=\sum_{p=1}^{n}\frac{\left(\Phi_{pa_{1}}-\Phi_{pa_{n}}\right)^2}{\mu_{p}+m};
\end{equation}
For networks $C^{m}_{P_n}$ and $C^{m}_{P_{n+1}}$, it follows from Lemma \ref{lem3.1} that
\begin{equation*}\lim_{n \rightarrow \infty}\big(R_{C^{m}_{P_{n+1}}}[a_1,a_{n+1}]-R_{C^{m}_{P_{n}}}[a_1,a_n]\big)=0.\end{equation*}
This means
\begin{equation}\label{EqA}\lim_{n \rightarrow \infty}\left(\sum_{p=1}^{n}\frac{\left(\Phi_{pa_{1}}-\Phi_{pa_{n}}\right)^2}{\mu_{p}+m}-\sum_{p=1}^{n-1}\frac{\left(\Psi_{pa_{1}}-\Psi_{pa_{n}}\right)^2}{\lambda_{p}+m}\right)=0.\end{equation}
On the other hand, again by Lemma \ref{lem3.6}, and due to the symmetry of network $C^m_{P_n}$, it is not hard to derive the following equations:
\begin{equation}\label{eq4.16}
4R_{C^{m}_{P_{n}}}[a_1,b]-R_{C^{m}_{P_n}}[a_1,a_n]-\frac{4}{mn}=\sum_{p=1}^{n-1}\frac{\left(\Psi_{pa_{1}}+\Psi_{pa_{n}}\right)^2}{\lambda_{p}+m};
\end{equation}
\begin{equation}
4R_{C^{m}_{P_{n+1}}}[a_1,b]-R_{C^{m}_{P_{n+1}}}[a_1,a_{n+1}]-\frac{4}{m(n+1)}=\sum_{p=1}^{n}\frac{\left(\Phi_{pa_{1}}+\Phi_{pa_{n}}\right)^2}{\mu_{p}+m}.
\end{equation}
Together with Lemmas \ref{lem3.1} and \ref{lem3.111}, we have
\begin{align}\label{EqB}
&\lim_{n \rightarrow \infty}\left(\sum_{p=1}^{n}\frac{\left(\Phi_{pa_{1}}+\Phi_{pa_{n}}\right)^2}{\mu_{p}+m}-\sum_{p=1}^{n-1}\frac{\left(\Psi_{pa_{1}}+\Psi_{pa_{n}}\right)^2}{\lambda_{p}+m}\right) \nonumber \\
&=\lim_{n \rightarrow \infty}4 \big(R_{C^{m}_{P_{n+1}}}[a_1,b]-R_{C^{m}_{P_{n}}}[a_1,b]\big)+  \lim_{n \rightarrow \infty}\big(R_{C^{m}_{P_{n}}}[a_1,a_{n}]-R_{C^{m}_{P_{n+1}}}[a_1,a_{n+1}]\big)  \nonumber \\
& +\lim_{n \rightarrow \infty}\big(\frac{4}{mn}-\frac{4}{m(n+1)}\big)=0.
\end{align}
In addition, note that the clique $K_2$ has two Laplacian eigenvalues $2$ and $0$ with respect to corresponding orthogonal eigenvectors:
$$(-\frac{\sqrt{2}}{2},\frac{\sqrt{2}}{2})^{T}, \ \ \ (\frac{\sqrt{2}}{2},\frac{\sqrt{2}}{2})^{T}.$$
By using Lemma \ref{lem2.6}, for any integer $r$, where $1\leq r \leq 2^{k}-1$, we know $\nu_r$ is an positive integer multiple of $2$ and
$$|\Upsilon_{rb_{i}}|=|\Upsilon_{rb_{j}}|=(\frac{\sqrt{2}}{2})^{r}.$$
Hence, together with Eqs. (\ref{EqA}) and (\ref{EqB}), no matter the choice of  $\Upsilon_{rb_{i}}$ and $\Upsilon_{rb_{j}}$, we obtain
$$\lim_{n \rightarrow \infty}\left(\sum_{p=1}^{n}\frac{\left(\Phi_{pa_{1}}\Upsilon_{rb_i}-\Phi_{pa_{n+1}}\Upsilon_{rb_j}\right)^2}{\mu_{p}+\nu_{r}}-\sum_{p=1}^{n-1}\frac{\left(\Psi_{pa_{1}}\Upsilon_{rb_i}-
\Psi_{pa_{n}}\Upsilon_{rb_j}\right)^2}{\lambda_{p}+\nu_{r}}\right)=0,$$
which yield that
\begin{equation}\label{F}
\lim_{n \rightarrow \infty}\sum_{q=1}^{2^{k}-1}\left(\sum_{p=1}^{n}\frac{\left(\Phi_{pa_{1}}\Upsilon_{qb_i}-\Phi_{pa_{n+1}}\Upsilon_{qb_j}\right)^2}{\mu_{p}+\nu_{q}}- \sum_{p=1}^{n-1}\frac{\left(\Psi_{pa_{1}}\Upsilon_{qb_i}-\Psi_{pa_{n}}\Upsilon_{qb_j}\right)^2}{\lambda_{p}+\nu_{q}}\right)=0.  \end{equation}
Consequently, according to  Eqs. (\ref{eq4.2}), (\ref{E}) and (\ref{F}), we have
$$\lim_{n \rightarrow \infty}\big(R_{U_{n+1}}[(a_{1},b_i),(a_{n+1},b_j)]-R_{U_{n}}[(a_{1},b_i),(a_{n},b_j)]\big)=\frac{1}{2^{k}}.$$
This completes the proof.\end{proof}
As a consequence of Theorem \ref{Tm4.1}, we can directly give the solution to Conjecture \ref{conj1}.
\begin{thm} \label{thm4.2}
Let $G_n=P_{n}\square C_4$. Then
$$\lim_{n \rightarrow \infty}\big(R_{G_{n+1}}[(a_{1},b_1),(a_{n+1},b_3)]-R_{G_{n}}[(a_{1},b_1),(a_{n},b_3)]\big)=\frac{1}{4}.$$
\end{thm}

Finally, we determine all resistance diametrical pairs in $G_n$ for $n\geq 2 $. The following result gives an monotonicity properties of resistance distances between $G_n$ and $G_{n+1}.$
\begin{lem} \label{lem4.6}
Let $G_n=P_{n}\square C_4$. Then
 \begin{equation*}
 R_{G_{n+1}}[(a_{1},b_1),(a_{n+1},b_3)]>R_{G_{n}}[(a_{1},b_1),(a_{n},b_3)].
 \end{equation*}
\end{lem}
\begin{proof}
Since $Q_{2}=C_{4}$ with $V(C_4)=\{b_1,b_2,b_3,b_4\}$ has three positive Laplacian eigenvalues $\nu_{1}=4$ and $\nu_{2}=\nu_{3}=2$ with respect to corresponding orthogonal eigenvectors:
$$\Upsilon_{1}=(\frac{1}{2},-\frac{1}{2},\frac{1}{2},-\frac{1}{2})^{T}, \ \Upsilon_{2}=(-\frac{1}{2},-\frac{1}{2},\frac{1}{2},\frac{1}{2})^{T}, \ \Upsilon_{3}=(-\frac{1}{2},\frac{1}{2},\frac{1}{2},-\frac{1}{2})^{T}.$$
Then using Theorem \ref{thm3.1}, we have
\begin{align}\label{eq4.1} R_{G_n}[(a_1,b_1),(a_n,b_3)]=\frac{n-1}{4}+\frac{1}{n}+\frac{1}{4}\sum_{p=1}^{n-1}\frac{\left(\Psi_{pa_{1}}-\Psi_{pa_{n}}\right)^2}{\lambda_{p}+4}+\frac{1}{2}\sum_{p=1}^{n-1}\frac{\left(\Psi_{pa_{1}}+\Psi_{pa_{n}}\right)^2}{\lambda_{p}+2}.
\end{align}
We continue to consider ladder graph $L_n$. Since $K_2$ has only one positive Laplacian eigenvalue $2$ with respect to eigenvector $(-\frac{\sqrt{2}}{2},\frac{\sqrt{2}}{2})^{T}$, again by Theorem \ref{thm3.1}, we have
\begin{equation}\label{eq4.3} R_{L_{n}}[(a_1,c_1),(a_n,c_2)]=\frac{n-1}{2}+\frac{1}{n}+\frac{1}{2}\sum_{p=1}^{n-1}\frac{\left(\Psi_{pa_{1}}+\Psi_{pa_{n}}\right)^2}{\lambda_{p}+2}.
\end{equation}
Consequently, according to Eqs. (\ref{eq20}), (\ref{eq4.1}) and (\ref{eq4.3}), we have
$$R_{G_{n}}[(a_{1},b_1),(a_{n},b_3)] = R_{L_{n}}[(a_{1},c_1),(a_{n},c_2)]+\frac{1}{4}R_{C^{4}_{P_{n}}}[a_1,a_n]-\frac{n-1}{4}.$$
Together with Lemmas \ref{lem3.3} and \ref{lem3.1}, we have
\begin{align}
&R_{G_{n+1}}[(a_{1},b_1),(a_{n+1},b_3)]- R_{G_{n}}[(a_{1},b_1),(a_{n},b_3)] \nonumber \\
&= \left(R_{L_{n+1}}[(a_{1},c_1),(a_{n+1},c_2)]-R_{L_{n}}[(a_{1},c_1),(a_{n},c_2)]\right)  \nonumber \\
& +\frac{1}{4}\big(R_{C^{4}_{P_{n+1}}}[a_1,a_{n+1}]-R_{C^{4}_{P_{n}}}[a_1,a_{n}]\big) -\frac{n}{4} +\frac{n-1}{4} \nonumber\\
&>\frac{1}{4}-\frac{1}{4}=0. \nonumber
\end{align}
This completes the proof.\end{proof}

\begin{thm} \label{thm4.3}
Let $G_n=P_{n}\square C_4$. Then
\begin{equation*}
RD(G_n)= \big\{[(a_1, b_1), (a_n, b_3)], [(a_1, b_2), (a_n, b_4)], [(a_1, b_3), (a_n, b_1)], [(a_1, b_4), (a_n, b_2)]\big\}.
\end{equation*}
\end{thm}
\begin{proof}For $n=2$, $G_{2}=K_2\square C_{4}$ be a cube $Q_{3}$. According to Eq. (\ref{D}), we know
$$D_{r}(G_2)=R_{G_2}[(a_1, b_1), (a_2, b_3)]=\frac{5}{6}.$$
Now we consider $n\geq3$ and prove the following three claims.
\renewcommand{\theclaim}{1}
\begin{claim}\label{cla4.5}
\begin{equation*}
R_{G_n}[(a_1, b_1), (a_n, b_1)] < R_{G_n}[(a_1, b_1), (a_n, b_3)].
\end{equation*}
\end{claim}
\emph{Proof of Claim \ref{cla4.5}.}
Simple calculations via Theorem \ref{thm3.1}, we have
\begin{align}\label{eq4.13}
&R_{G_n}[(a_1,b_1),(a_n,b_3)]-R_{G_n}[(a_1,b_1),(a_n,b_1)] \nonumber \\
&= R_{L_n}[(a_1,c_1),(a_n,c_2)]-R_{L_n}[(a_1,c_1),(a_n,c_1)] =\frac{1}{n}+2\sum_{p=1}^{n-1}\frac{\Psi_{pa_{1}}\Psi_{pa_{n}}}{\lambda_{p}+2}.
\end{align}
Then according to Lemma \ref{lem3.4}, we have
\begin{equation*}\label{eq4.14}
R_{L_{n}}[(a_1,c_1),(a_{n},c_2)]-R_{L_{n}}[(a_1,c_1),(a_{n},c_1)]>0.
\end{equation*}
It follows that
\begin{equation*}
R_{G_n}[(a_1, b_1), (a_n, b_3)] > R_{G_n}[(a_1, b_1), (a_n, b_1)].
\end{equation*}
Hence Claim \ref{cla4.5} holds.\\
 \renewcommand{\theclaim}{2}
 \begin{claim}\label{cla4.6}
\begin{equation*}
R_{G_n}[(a_1, b_1), (a_n, b_2)] < R_{G_n}[(a_1, b_1), (a_n, b_3)].
\end{equation*}
\end{claim}
\emph{Proof of Claim \ref{cla4.6}.}
Similarly, by simple calculation of Theorem \ref{thm3.1}, we show
\begin{equation}\label{eq4.12} R_{G_n}[(a_1,b_1),(a_n,b_3)]-R_{G_n}[(a_1,b_1),(a_n,b_2)]=\frac{1}{4n}+\sum_{p=1}^{n-1}\frac{\Psi_{pa_{1}}\Psi_{pa_{n}}}{\lambda_{p}+2}-\sum_{p=1}^{n-1}\frac{\Psi_{pa_{1}}\Psi_{pa_{n}}}{\lambda_{p}+4}.
\end{equation}
Comparing Eqs. (\ref{eq20}) and (\ref{eq4.16}), then
\begin{equation}\label{eq4.133}
 2R_{C^{4}_{P_{n}}}[a_1,b]-R_{C^{4}_{P_{n}}}[a_1,a_n]=\frac{1}{2n}+2\cdot\sum_{p=1}^{n-1}\frac{\Psi_{pa_{1}}\Psi_{pa_{n}}}{\lambda_{p}+4}.
\end{equation}
Consequently, according to Eqs. (\ref{eq4.13})-(\ref{eq4.133}), we get
\begin{align}\label{eq4.20}
&R_{G_n}[(a_1,b_1),(a_n,b_3)]-R_{G_n}[(a_1,b_1),(a_n,b_2)] \nonumber \\
&= \frac{1}{2}\left(R_{L_n}[(a_1,c_1),(a_n,c_2)]-R_{L_n}[(a_1,c_1),(a_n,c_1)]\right)-\frac{1}{2}(2R_{C^{4}_{P_{n}}}[a_1,b]-R_{C^{4}_{P_{n}}}[a_1,a_n]).
\end{align}
By Lemma \ref{lem3.4}, for network $L_n$ we have
\begin{equation}\label{eq4.22}
R_{L_{n}}[(a_1,c_1),(a_{n},c_2)]-R_{L_{n}}[(a_1,c_1),(a_{n},c_1)]>\frac{2\sqrt{3}}{(2+\sqrt{3})^{n}}.
\end{equation}
For network $C^{4}_{P_n}$, it follows directly from Lemma \ref{lem3.11} that
\begin{equation}\label{eq4.24}
2R_{C^{4}_{P_{n}}}[a_1,b]-R_{C^{4}_{P_{n}}}[a_1,a_n]<\frac{2}{4^{n}}.
\end{equation}
Substituting Eqs. (\ref{eq4.22}) and (\ref{eq4.24}) into Eq. (\ref{eq4.20}), we have
\begin{equation*}
R_{G_n}[(a_1,b_1),(a_n,b_3)]>R_{G_n}[(a_1,b_1),(a_n,b_2)].
\end{equation*}
Thus Claim \ref{cla4.6} is proved.
\renewcommand{\theclaim}{3}
\begin{claim}\label{cla4.3}
For $0\leq j-i \leq n-2$ and $1\leq p,q \leq 4$, then
$$R_{G_n}[(a_i, b_p), (a_j, b_q)] < R_{G_n}[(a_1, b_1), (a_n, b_3)].$$
\end{claim}

\textit{Proof of Claim \ref{cla4.3}.} For convenience, we distinguish the following two cases.\\
\textit{Case 1.} $j-i=0$. \\In this case, note that cube $G_2$ be a subnetwork of $G_n$. Then by Proposition \ref{prop2.4}, we have
 \begin{equation}\label{eq4.6}
 R_{G_n}[(a_i, b_p), (a_i, b_q)] \leq R_{G_2}[(a_1, b_p), (a_1, b_q)].
  \end{equation}
Since  vertex pair $[(a_1, b_1), (a_2, b_3)]$ reach the resistance diameter of $G_2$, we have
  \begin{equation}\label{eqaa}
  R_{G_2}[(a_1, b_p), (a_1, b_q)]< R_{G_2}[(a_1, b_1), (a_2, b_3)].
  \end{equation}
Then by Lemma \ref{lem4.6}, for $n\geq3$, we have
 \begin{equation}\label{eq4.7}
 R_{G_{2}}[(a_{1},b_1),(a_{2},b_3)] < R_{G_{n}}[(a_{1},b_1),(a_{n},b_3)].
 \end{equation}
Comparing Eqs. (\ref{eq4.6})-(\ref{eq4.7}), we get
 \begin{equation*}
 R_{G_n}[(a_i, b_p), (a_i, b_q)] < R_{G_n}[(a_1, b_1), (a_n, b_3)].
 \end{equation*}
 \textit{Case 2.} $0 <j-i \leq n-2$. \\ Considering network $G_{j-i+1}$ be a subnetwork of $G_n$. Again by Proposition \ref{prop2.4}, we have
 \begin{equation}\label{eq4.8}
 R_{G_n}[(a_i, b_p), (a_j, b_q)] \leq  R_{G_{j-i+1}}[(a_1, b_p), (a_{j-i+1}, b_q)].
 \end{equation}
Together with Claim \ref{cla4.5} and \ref{cla4.6}, we can see that
 \begin{equation}\label{eqbb}
 R_{G_{j-i+1}}[(a_1, b_p), (a_{j-i+1}, b_q)]\leq  R_{G_{j-i+1}}[(a_1, b_1), (a_{j-i+1}, b_3)].
 \end{equation}
Since $j-i+1 < n$, by Lemma \ref{lem4.6}, we have
 \begin{equation}\label{eq4.9}
 R_{G_{j-i+1}}[(a_{1},b_1),(a_{j-i+1},b_3)]< R_{G_{n}}[(a_{1},b_1),(a_{n},b_3)].
 \end{equation}
 Comparing Eqs. (\ref{eq4.8})-(\ref{eq4.9}), it follows that
 \begin{equation*}
 R_{G_n}[(a_i, b_p), (a_j, b_q)] < R_{G_n}[(a_1, b_1), (a_n, b_3)].
 \end{equation*}
 Thus Claim \ref{cla4.3} is proved.\\
According to the symmetry of the network $G_n$, it is not to hard seen that
$$ R_{G_n}[(a_1, b_1), (a_n, b_3)] = R_{G_n}[(a_1, b_2), (a_n, b_4)]=R_{G_n}[(a_1, b_3), (a_n, b_1)]=R_{G_n}[(a_1, b_4), (a_n, b_2)].$$
By Claims \ref{cla4.5}-\ref{cla4.3}, we know that for $i, j\in \{1,2,...,n\}$ and $p, q \in \{1,2,3,4\}$, if
  \begin{align*}
[(a_i, b_p), (a_j, b_q)]\notin & \big\{ [(a_1, b_1), (a_n, b_3)], [(a_1, b_2), (a_n, b_4)],\\
   &[(a_1, b_3), (a_n, b_1)], [(a_1, b_4), (a_n, b_2)]\big\},
\end{align*}
then
$$R_{G_n}[(a_i, b_p), (a_j, b_q)] < R_{G_n}[(a_1, b_1), (a_n, b_3)].$$
Thus it follows that
\begin{equation*}
RD(G_n)= \big\{[(a_1, b_1), (a_n, b_3)], [(a_1, b_2), (a_n, b_4)], [(a_1, b_3), (a_n, b_1)], [(a_1, b_4), (a_n, b_2)]\big\}.
\end{equation*}
The proof is complete. \end{proof}
Finally, according to Theorems \ref{thm4.2} and \ref{thm4.3}, we could give an equivalent explanation of the Conjecture \ref{conj1} from the direction on resistance diameter.
\begin{cor}
 Let $G_n=P_{n}\square C_4$. Then
$$\lim_{n \rightarrow \infty}\big(D_r(G_{n+1})-D_r(G_{n}) \big)=\frac{1}{4}.$$
\end{cor}

\section{Date availability}
No data was used for the research described in the article.
\section{Acknowledgments}
The support of the National Natural Science Foundation of China (through grant no. 12171414, 11571155) and Taishan Scholars Special Project of Shandong Province are greatly acknowledged.
\bibliographystyle{plain}

\begin{thebibliography}{9}
{\small

\bibitem{djk4} D.J. Klein, M. Randi\'{c}, Resistance distance. J. Math. Chem. 12 (1993) 81--95.

\bibitem{llo}L. Lov\'{a}sz, K. Vesztergombi, Geometric representations of graphs, in: Paul Erd\H{o}s and his Mathematics, Proc. Conf. Budapest, 1999.

\bibitem{aaz2}A.Azimi, R.B. Bapat, D.G. MFarrokhi, Resistance distance of blowups of trees, Discrete. Math. 344(7) (2021) 112387.

\bibitem{jge}J. Ge, F. Dong, Spanning trees in complete bipartite graphs and resistance distance in nearly complete bipartite graphs, Discrete Appl. Math. 283 (2020) 542--554.

\bibitem{ysun2}Y. Yang, W. Sun, Minimal hexagonal chains with respect to the Kirchhoff index, Discrete Math. 345 (2022) 113099.

\bibitem{ppm2}P.P. Mondal, R.B. Bapat, F. Atik, On the inverse and Moore-Penrose inverse of resistance matrix of graphs with more general matrix weights, J. Appl. Math. Comput. 69(6) (2023) 4805--4820..

\bibitem{agh} A. Ghosh, S. Boyd, A. Saberi, Minimizing effective resistance of a graph, SIAM Rev. 50(1) (2008) 37--66.

\bibitem{wba}W. Barrett, E.J. Evans, A.E. Francis, Resistance distance in straight linear 2-trees, Discrete Appl. Math. 258 (2019) 13--34.

\bibitem{eje}E.J. Evans, A.E. Francis, Algorithmic techniques for finding resistance distances on structured graphs, Discrete Appl. Math. 320 (2022) 387--407.

\bibitem{kde}K. Devriendt, A. Ottolini, S. Steinerberger, Graph curvature via resistance distance, Discrete Appl. Math. 348 (2024) 68--78.

\bibitem{wsa}W. Sajjad, X. Pan, Computation of resistance distance with Kirchhoff index of body centered cubic structure, J. Math. Chem. 62 (2024) 902--921.

\bibitem{jli}J. Liu, X. Wang, J. Cao, The Coherence and Properties Analysis of Balanced 2p-Ary Tree Networks, IEEE T. Netw. Sci. Eng. (2024). doi: 10.1109/TNSE.2024.3395710.

\bibitem{ysh}Y. Sheng, Z. Zhang, Low-mean hitting time for random walks on heterogeneous networks, IEEE T. Inform. Theory.  65(11) (2019) 6898--6910.

\bibitem{mss} M.S. Sardar, H. Hua, X. Pan, H. Raza, On the resistance diameter of hypercubes, Physica A 540 (2020) 123076.

\bibitem{yli} Y. Li, S. Xu, H. Hua, X. Pan, On the resistance diameter of the Cartesian and lexicographic product of paths, J. Appl. Math. Comput. 68 (2022) 1743--1755.

\bibitem{kns} M. Knor, L. Niepel, L. \v{S}olt\'{e}s, Centers in line graphs, Math. Slovaca. 43 (1993) 11--20.

\bibitem{sxu} S. Xu, Y. Li, H. Hua, X. Pan, On the resistance diameters of graphs and their line graphs, Discrete Appl. Math. 306 (2022) 174--185.

\bibitem{wsu} W. Sun, Y. Yang, Solution to a conjecture on resistance diameter of lexicographic product of paths, Discrete Appl. Math. 337 (2023) 139--148.

\bibitem{mva}M. Vaskouski, H. Zadarazhniuk, Resistance diameters and critical probabilities of Cayley graphs on irreducible complex reflection groups, J. Algebr. Comb. 59 (2024) 621--634.

\bibitem{aek} A.E. Kennelly, The equivalence of triangles and three-pointed stars in conducting networks, Electri. World Engineer. 34 (1899) 413--414.

\bibitem{djk2} D.J. Klein, Resistance-distance sum rules, Croat. Chem. Acta. 75 (2002) 633--649.

\bibitem{svg} S.V. Gervacio, Resistance distance in complete $n$-partite graphs, Discrete Appl. Math. 203 (2016) 53--61.

\bibitem{sab}G. Sabidussi: Graph multiplication. Math. Z. 72 (1959) 446--457.

\bibitem{dbw}D.B. West, Introduction to graph theory, Upper Saddle River: Prentice hall. (2001).

\bibitem{aaz1}A.A Zykov, On some properties of linear complexes, Matematicheskii sbornik. 66(2) (1949) 163--188.

\bibitem{caa}C.A. Alfaro, A. Arroyo, Der\v{n}\'{a}r M, et al, The crossing number of the cone of a graph, SIAM J. Discrete Math. 32(3) (2018) 2080--2093.

\bibitem{jst} J. Strutt, On the theory of resonance, Philos. Trans. Roy. Soc. A.  161 (1871) 77--118.

\bibitem{rme}R. Merris, Laplacian graph eigenvectors, Linear Algebra. Appl. 278(1-3) (1998) 221--236.

\bibitem{yya1} Y. Yang, D.J. Klein, Resistance distances in composite graphs, J. Phys. A Math. Theor. 47(37) (2014) 375203.

\bibitem{zci}Z. Cinkir, Effective resistances and Kirchhoff index of ladder graphs, J. Math. Chem. 54 (2016) 955--966.

\bibitem{lsh}L. Shi, H. Chen, Resistance distances in the linear polyomino chain, J. Appl. Math. Comput. 57(1-2) (2018) 147--160.

\bibitem{fha}F. Harary, J.P. Hayes, H. Wu, A survey of the theory of hypercube graphs, Comput. Math. Appl. 15(4) (1988) 277--289.

\bibitem{mli}M. Livingston, Q.F. Stout, Embeddings in hypercubes, Math. Comput. Model. 11 (1988) 222--227.

\bibitem{hjb}H.J. Bandelt, Retracts of hypercubes, J. Graph Theory. 8(4) (1984) 501--510.

\bibitem{lsc}L.S. Chandran, T. Kavitha, The treewidth and pathwidth of hypercubes, Discrete Math.  306(3) (2006) 359--365.

\bibitem{jva}J. Vandenbussche, D.B. West, Matching extendability in hypercubes, SIAM J. Discrete Math. 23(3) (2009) 1539--1547.

\bibitem{wol}Wolfram Research, Inc. Mathematica. version 12.0. Champaign, IL: Wolfram research Inc. (2019).

}
\end{thebibliography}

\end{CJK}
\end{document}